\ifx \processToSlides \undefined
  \ifx\usepackage\RequirePackage    
    \let\usingAmsArtXII\usepackage  
  \fi
\else
  \documentclass[12pt]{amsart}
  \usepackage{amssymb,amscd}
  \usepackage{amsfonts,amsmath,amssymb}
\usepackage{hyperref}
\usepackage{latexsym,fullpage,amsfonts,amssymb,amsmath,amscd,graphics,epic,enumerate}

  \usepackage[letterpaper,margin=2cm,includeheadfoot]{geometry}
  \def \useHugeSize {}
  \def \numberingIsThrough {}
\fi

\ifx \labelsONmargin\undefined

  \ifx\RequirePackage\undefined
    \expandafter\ifx\csname amsart.sty\endcsname\relax
        \documentstyle[12pt,amssymb,amscd]{amsart}
    \fi

    \def\atSign{@@}

    \def\mathbb{\Bbb}
    \def\mathfrak{\frak}
    \def\mathbf{\bold}
    \ifx\boldsymbol\undefined   
      \def\boldsymbol#1{{\bold #1}}
    \fi

    \def\mathbit{\boldsymbol}

    \makeatletter
    \newenvironment{proof}{%
         \@ifnextchar[{%
                       \expandafter\let\expandafter\end@proof
                         \csname endpf*\endcsname
                         \my@proof
                      }{\let\end@proof\endpf\pf}%
        }{\end@proof}
    \def\my@proof[#1]{\@nameuse{pf*}{#1}}
    \makeatother

    \def\xrightarrow[#1]#2{@>{#2}>{#1}>}
    \def\xleftarrow[#1]#2{@<{#2}<{#1}<}

    \def\providecommand#1{\def#1}
    \def\emph#1{{\em #1}}
    \def\textbf#1{{\bf #1}}

    \def\mathring{\overset{\,\,{}_\circ}}
  \else
      \ifx\usepackage\RequirePackage
      \else
        \documentclass[12pt]{amsart}
        \usepackage{amssymb,amscd}
        \usepackage[all]{xy}
  \usepackage{xcolor}
    \let\usingAmsArtXII\usepackage
      \fi
      \ifx \mathring\undefined      
        \DeclareMathAccent{\mathring}{\mathalpha}{operators}{"17}
      \fi

      \long\def\FAKEendPROOF{\endtrivlist}
      \ifx\endproof\FAKEendPROOF
      \def\endproof{\qed\endtrivlist}
      \fi

      \ifx\mathbit\undefined    
        \DeclareMathAlphabet{\mathbit}{OML}{cmm}{b}{it}
      \fi

      \def\atSign{@}

      \ifx \usingAmsArtXII \undefined
      \else             
    \advance\oddsidemargin -0.1\textwidth
    \evensidemargin\oddsidemargin
      \fi

      \def\Sb#1\endSb{_{\substack{#1}}}
      \def\Sp#1\endSp{^{\substack{#1}}}

  \fi
\makeatletter
\@ifundefined{DeclareFontShape}
     {\@ifundefined{selectfont}
             {
                \message{We need AmS-LaTeX, this version of LaTeX does not
                        support it}\@@end
             }{
                \def\mathcal{\cal}
                
                \def\pcyr{%
                        \def\default@family{UWCyr}%
                        \let\oldSl@\sl
                        \def\sl{\def\default@shape{it}\oldSl@}%
                        \cyracc
                        \language\Russian\family{UWCyr}\selectfont
                }
             }}{
                \DeclareFontEncoding{OT2}{}{\noaccents@}
                \DeclareFontSubstitution{OT2}{cmr}{m}{n}
                \DeclareFontFamily{OT2}{cmr}{\hyphenchar\font45 }
                \DeclareFontShape{OT2}{cmr}{m}{n}{%
                     <5><6><7><8><9><10>gen*wncyr %
                     <10.95><12><14.4><17.28><20.74><24.88> wncyr10 %
                }{}
                \DeclareFontShape{OT2}{cmr}{m}{it}{%
                     <5><6><7><8><9><10> gen * wncyi%
                     <10.95><12><14.4><17.28><20.74><24.88> wncyi10%
                }{}
                \DeclareFontShape{OT2}{cmr}{bx}{n}{%
                     <5><6><7><8><9><10> gen * wncyb%
                     <10.95><12><14.4><17.28><20.74><24.88> wncyb10%
                }{}
                \DeclareFontShape{OT2}{cmr}{m}{sl}{%
                     <-> ssub * cmr/m/it%
                }{}
                \DeclareFontShape{OT2}{cmr}{m}{sc}{%
                     <5><6><7><8><9><10>%
                     <10.95><12><14.4><17.28><20.74><24.88> wncysc10%
                }{}
                \DeclareFontFamily{OT2}{cmss}{\hyphenchar\font45 }
                \DeclareFontShape{OT2}{cmss}{m}{n}{%
                     <8><9><10> gen * wncyss%
                     <10.95><12><14.4><17.28><20.74><24.88> wncyss10%
                }{}
                
                \def\cyrencodingdefault{OT2}
                \def\pcyr{%
                        \cyracc
                        \let\encodingdefault\cyrencodingdefault
                        \language\Russian\fontencoding{OT2}\selectfont
                }
             }

\@ifundefined{theorembodyfont}
     {
        \def\theorembodyfont#1{\relax}
        \makeatletter
          \let\@@th@plain\th@plain
          \def\th@plain{ \@@th@plain \slshape }
        \makeatother
        \let\normalshape\relax
     }{}

\ifx\cprime\undefined
     \def\cprime{$'$}
\fi
\makeatletter
  \def\@sect@my#1#2#3#4#5#6[#7]#8{%
\ifnum #2>\c@secnumdepth
   \let\@svsec\@empty
 \else
   \refstepcounter{#1}%
\edef\@svsec{\ifnum#2<\@m
             \@ifundefined{#1name}{}{\csname #1name\endcsname\ }\fi
\noexpand\rom{\csname the#1\endcsname.}\enspace}\fi
 \@tempskipa #5\relax
 \ifdim \@tempskipa>\z@ 
   \begingroup #6\relax
   \@hangfrom{\hskip #3\relax\@svsec}{\interlinepenalty\@M #8\par}%
   \endgroup
   \if@article\else\csname #1mark\endcsname{%
        \ifnum \c@secnumdepth >#2\relax\csname the#1\endcsname. \fi#7}\fi
\ifnum#2>\@m \else
       \let\@tempf\\ \def\\{\protect\\}\addcontentsline{toc}{#1}%
{\ifnum #2>\c@secnumdepth \else
             \protect\numberline{%
               \ifnum#2<\@m
               \@ifundefined{#1name}{}{\csname #1name\endcsname\ }\fi
               \csname the#1\endcsname.}\fi
           #8}\let\\\@tempf
     \fi
 \else
  \def\@svsechd{#6\hskip #3\@svsec
    \@ifnotempty{#8}{\ignorespaces#8\unskip
       \ifnum\spacefactor<1001.\fi}%
        \ifnum#2>\@m \else
          \let\@tempf\\ \def\\{\protect\\}\addcontentsline{toc}{#1}%
            {\ifnum #2>\c@secnumdepth \else
              \protect\numberline{%
                \ifnum#2<\@m
                \@ifundefined{#1name}{}{\csname #1name\endcsname\ }\fi
                \csname the#1\endcsname.}\fi
             #8}\let\\\@tempf\fi}%
 \fi
\@xsect{#5}}
  \ifx\RequirePackage\undefined
  \let\@sect\@sect@my             
  \fi
  \ifx\RequirePackage\undefined 
  \def\th@remark@my{\theorempreskipamount6\p@\@plus6\p@
    \theorempostskipamount\theorempreskipamount
    \def\theorem@headerfont{\it}\normalshape}
    \let\th@remark\th@remark@my
  \else             
    \let\o@@remark\th@remark

    \ifx\theorempostskipamount\undefined        
    \else                       
      \def\th@remark{\o@@remark
    \ifdim\theorempostskipamount < 2pt\relax
      \theorempostskipamount\theorempreskipamount
         \multiply\theorempostskipamount\tw@
         \divide\theorempostskipamount\thr@@
    \fi
      }
    \fi
  \fi
\let\myLabel\@gobble
\def\labelsONmargin{\@mparswitchfalse\def\myLabel##1{\@bsphack\marginpar
                                  {\normalshape\tiny\rm Label ##1}\@esphack}}
\ifx \url\undefined
  \def\url#1{{\tt #1}}%
\fi
\def\cyracc{\def\u##1{
                \if \i##1\char"1A%
                \else \if I##1\char"12%
                \else \accent"24 ##1\fi\fi }%
\def\"##1{\if e##1{\char"1B}%
                \else \if E##1{\char"13}%
                \else \accent"7F ##1\fi\fi }%
\def\9##1{\if##1z\char"19
\else\if##1Z\char"11
\else\if##1E\char"03
\else\if##1e\char"0B
\else\if##1u\char"18
\else\if##1U\char"10
\else\if##1A\char"17
\else\if##1a\char"1F
\else\if##1p\char"7E
\else\if##1P\char"5E
\else\if##1Q\char"5F
\else\if##1q\char"7F
\else\if##1i\char"1A
\else\if##1I\char"12
\else\if##1N\char"7D
\fi
\fi
\fi
\fi
\fi
\fi
\fi
\fi
\fi
\fi
\fi
\fi
\fi
\fi
\fi
}%
\def\cydot{{\kern0pt}}}%
\def\cydot{$\cdot$}

\ifx\Russian\undefined
        \def\Russian{0\relax
    \message{Don't know the hyphenation rules for Russian^^J
                        Please do INITeX with `input  russhyph' in the
                        command line}%
                \gdef\Russian{0\relax}%
        }
\fi
\ifx\RequirePackage\undefined           
  \def\@putname#1#2#3#4{\def\@@ref{#3}\let\old@bf\bf
        \def\bf##1{\old@bf\if?\noexpand##1?{#4}\else##1\fi}%
    #1{#2}%
        \let\bf\old@bf}
\else                       
  \def\@putname#1#2#3#4{\def\@@ref{#3}\let\old@bf\bf    
    \let\old@reset@font\reset@font          
        \def\bf##1{\old@bf\if?\noexpand##1?{#4}\else##1\fi}%
    \def\reset@font##1##2{\old@reset@font##1\if?\noexpand##2?{#4}\else##2\fi}#1{#2}%
        \let\bf\old@bf\let\reset@font\old@reset@font}
\fi
\let\my@ref=\ref
\def\ref#1{\@putname\my@ref{#1}{#1}{\tiny\rm\@@ref}}
\let\my@pageref=\pageref
\def\pageref#1{\@putname\my@pageref{#1}{#1}{\tiny\rm\@@ref}}
\let\my@cite=\cite
\def\cite#1{\@putname\my@cite{#1}{\@citeb}{\tiny\rm\@@ref}}
\makeatother
\fi
\relax
\theoremstyle{plain} 
\theorembodyfont{\sl}

\ifx\useDoubleSpacing\undefined
\else
  \usepackage{setspace}
\fi

\ifx \numberingIsThrough \undefined
\numberwithin{equation}{section}
\fi
\theorembodyfont{\rm}
\theoremstyle{definition}

\ifx \numberingIsThrough \undefined
\newtheorem{definition}{Definition}[section]
\else
\newtheorem{definition}{Definition}
\fi

\newtheorem{conjecture}[definition]{Conjecture}
\newtheorem{example}[definition]{Example}

\theoremstyle{remark}

\newtheorem{remark}[definition]{Remark} 

\ifx \numberingIsThrough \undefined
\newtheorem{note}{Note}[section] 
\newtheorem{summary}{Summary}[section] 
\else

\fi

\theoremstyle{plain} 
\theorembodyfont{\sl}

\newtheorem{theorem}[definition]{Theorem}
\newtheorem{lemma}[definition]{Lemma}
\newtheorem{corollary}[definition]{Corollary}
\newtheorem{proposition}[definition]{Proposition}

\begin{document}
\bibliographystyle{amsplain}

\ifx\useHugeSize\undefined
\else
\Huge
\fi

\relax
\renewcommand{\v}{\varepsilon} \newcommand{\p}{\rho}
\newcommand{\m}{\mu}
\def\im{{\rm Im}}
\def\ker{{\rm Ker}}
\def\End{{\rm End}}
\def\Pic{{\bf Pic}}
\def\re{{\bf re}}
\def\e{{\bf e}}
\def\a{\alpha}
\def\ve{\varepsilon}
\def\b{\beta}
\def\D{\Delta}
\def\d{\delta}
\def\f{{\varphi}}
\def\ga{{\gamma}}
\def\L{\Lambda}
\def\lo{{\bf l}}
\def\s{{\bf s}}
\def\A{{\bf A}}
\def\B{{\bf B}}
\def\cB{{\mathcal {B}}}
\def\C{{\mathbb C}}
\def\F{{\bf F}}
\def\G{{\mathfrak {G}}}
\def\g{{\mathfrak {g}}}
\def\b{{\mathfrak {b}}}
\def\q{{\mathfrak {q}}}
\def\f{{\mathfrak {f}}}
\def\k{{\mathfrak {k}}}
\def\l{{\mathfrak {l}}}
\def\m{{\mathfrak {m}}}
\def\n{{\mathfrak {n}}}
\def\o{{\mathfrak {o}}}
\def\p{{\mathfrak {p}}}
\def\s{{\mathfrak {s}}}
\def\t{{\mathfrak {t}}}
\def\r{{\mathfrak {r}}}
\def\z{{\mathfrak {z}}}
\def\h{{\mathfrak {h}}}
\def\H{{\mathcal {H}}}
\def\O{\Omega}
\def\M{{\mathcal {M}}}
\def\T{{\mathcal {T}}}
\def\N{{\mathcal {N}}}
\def\U{{\mathcal {U}}}
\def\Z{{\mathbb Z}}
\def\P{{\mathcal {P}}}
\def\GVM{ GVM }
\def\iff{ if and only if  }
\def\add{{\rm add}}
\def\deg{{\rm deg}}
\def\Hom{{\rm Hom}}
\def\ld{\ldots}
\def\vd{\vdots}
\def\sl{{\rm sl}}
\def\mod{{\rm mod}}
\def\len{{\rm len}}
\def\cd{\cdot}
\def\dd{\ddots}
\def\q{\quad}
\def\qq{\qquad}
\def\ol{\overline}
\def\tl{\tilde}
\def\nn{\nonumber}

\title [Representations of principal $W$-algebra]{Representations of principal $W$-algebra for the superalgebra $Q(n)$ and the super Yangian $YQ(1)$}

\author{ Elena Poletaeva and Vera Serganova }

\date{ \today }

\address{School of Mathematical and Statistical Sciences, University of Texas Rio Grande
Valley, Edinburg, TX 78539} \email{elena.poletaeva\atSign{}utrgv.edu}
\address{ Dept. of Mathematics, University of California at Berkeley,
Berkeley, CA 94720 } \email{serganov\atSign{}math.berkeley.edu}

\maketitle

\begin{abstract} We classify irreducible representations of finite $W$-algebra for the queer Lie superalgebra $Q(n)$ associated with the principal
  nilpotent coadjoint orbits. We use this classification and our previous results to obtain a classification of irreducible finite-dimensional representations of
  the super Yangian
  $YQ(1)$.
\end{abstract}

\section{Introduction}
In the classical case a finite $W_e$-algebra is a quantization of the
Slodowy slice to the adjoint orbit of a nilpotent element $e$ of a
semisimple Lie
algebra $\g$. Finite-dimensional simple $W_e$-modules are used for
classification of primitive ideals of $U(\g)$. Losev's work gives a new point of view on this classification, \cite{L1, L2, L3}.

In the supercase the theory of the primitive ideals is even more
complicated, \cite{CM}. It is interesting to generalize Losev's
result to the supercase. One step in this direction is to study
representations of finite $W$-algebras for a Lie superalgebra $\g$.
In the case when $\g=\mathfrak{gl}(m|n)$ and $e$ is the even principal nilpotent,
Brown, Brundan and Goodwin  classified irreducible
representations of $W_e$ and explored the connection with the category
$\mathcal O$ for $\g$ using coinvariants functor,  \cite{BBG, BG}.

In this paper, we study representations of finite $W$-algebras for the
Lie superalgebra $Q(n)$ associated with the principal even nilpotent coadjoint orbit.
Note that in this case the Cartan subalgebra $\h$ of $\g=Q(n)$ is not
abelian and contains a non-trivial odd part. By our previous
results (\cite{PS2}), we realize $W$ as a subalgebra of the universal
enveloping algebra $U(\h)$. One of the main results of the paper is a
classification of simple $W$-modules given in Theorem \ref{irreducible} (they are all finite-dimensional
by \cite{PS2}). The technique we use is completely different from
one used in \cite{BBG,BG} due to the lack of triangular decomposition of $W$ in
our case. Instead, we can describe the restriction of simple
$U(\h)$-modules to $W$
and prove that any simple $W$-module occurs as a constituent of this
restriction.

Next we proceed to classification of simple  finite-dimensional modules over the super Yangian $YQ(1)$ associated with the Lie superalgebra $Q(1)$.
The Yangians of type $Q$ were introduced by Nazarov in \cite{N1} and \cite{N}. In \cite{NS} these super Yangians were realized as limits of
certain centralizers in the universal enveloping algebras of type $Q$. We have shown previously in \cite{PS2} that a principal  $W$-algebra
(for any $n$) is a quotient of $YQ(1)$. Hence a simple module over a $W$-algebra can be lifted to a simple $YQ(1)$-module. However, not every simple
$YQ(1)$-module can be obtained in this way. We prove that any simple finite-dimensional $YQ(1)$-module is isomorphic to the tensor product of a module lifted
from a $W$-algebra and some one-dimensional module (Theorem \ref{mainYangian}).
We also obtain a formula for a generating function for the central character of a simple module.
This generating function is rational and probably should be considered as an analogue of the Drinfeld polynomial, see \cite{Mol} chapters 3, 4.

We also plan
in a subsequent paper to study the coinvariants functor from the category
$\mathcal O$ for $Q(n)$ to the category of $W$-modules.

As M. L. Nazarov pointed to us, it is also interesting to generalize the results of \cite{KN} to the case of $YQ(1)$ using the centralizer construction of $YQ(n)$ given in \cite{NS}.

\section{Notations and preliminary results}

We work in the category of super vector spaces over $\mathbb C$. All tensor products are over $\mathbb C$ unless specified otherwise.
By $\Pi$ we denote the functor of parity switch $\Pi(X)=X\otimes \mathbb C^{0|1}$.

Recall that if  $X$ is a simple finite-dimensional $\mathcal A$-module for some associative superalgebra $\mathcal A$, then
$\operatorname{End}_{\mathcal A}(X)=\mathbb C$ or $\operatorname{End}_{\mathcal A}(X)=\mathbb C[\epsilon]/(\epsilon^2-1)$, where the odd element
$\epsilon$ provides an $\mathcal A$ isomorphism $X\rightarrow \Pi(X)$. We say that $X$ is of M-type in the former case and of Q-type in the latter (see \cite{K1, CW}).

If $X$ and $Y$ are two simple
modules over associative superalgebras ${\mathcal A}$ and ${\mathcal B}$, we define the
$\mathcal A\otimes \mathcal B$-module $X\boxtimes Y$ as the usual tensor product
if at least one of $X$, $Y$ is of M-type and the tensor product over $\mathbb C[\epsilon]$
 if both $X$ and $Y$ are of Q-type.

In this paper we consider the Lie superalgebra $\g=Q(n)$ defined as follows (see \cite{K}). Equip  $\mathbb C^{n|n}$ with the odd operator $\zeta$ such
that $\zeta^2=-\operatorname{Id}$.
Then $Q(n)$ is the centralizer of $\zeta$ in the Lie superalgebra $\mathfrak{gl}(n|n)$. It is easy to see that $Q(n)$ consists of matrices
of the form $\left(\begin{matrix}A&B\\B&A\end{matrix}\right)$ where $A,B$ are $n\times n$-matrices. We fix the  Cartan subalgebra $\h\subset\g$ to be the set
of matrices
with diagonal $A$ and $B$. By $\mathfrak n^+$ (respectively, $\mathfrak n^-$) we denote the nilpotent subalgebras consisting of matrices with
strictly upper triangular (respectively, low triangular) $A$ and $B$.
The Lie superalgebra $\g$ has the triangular decomposition $\g=\mathfrak{n}^-\oplus\h\oplus \mathfrak{n}^+$ and we set $\b=\mathfrak{n}^+\oplus\h$.

Denote by $W$ the finite $W$-algebra associated with a principal\footnote{There is a unique open orbit in the nilpotent cone of the coadjoint
  representation, elements of this orbit are called principal.} even nilpotent element $\varphi$ in the coadjoint representation of $Q(n)$. Let us recall
the definition (see \cite{Pr1}).
Let $\{e_{i,j},f_{i,j}\,|\,i,j=1,\dots,n\}$ denote the basis consisting of elementary even and odd matrices. Choose $\varphi\in\g^*$ such that
$$\varphi(f_{i,j})=0,\quad \varphi(e_{i,j})=\delta_{i,j+1}.$$
Let $I_\varphi$ be the left ideal in $U(\g)$ generated by $x-\varphi(x)$ for all $x\in\mathfrak{n}^-$.
Let $\pi: U(\g)\rightarrow U(\g)/I_\varphi$ be the natural projection.
Then
$$W=\{\pi(y)\in U(\g)/I_\varphi\,|\,\operatorname{ad}(x)y\in I_\varphi\,\,\text{for all}\,\,x\in\mathfrak{n}^-\}.$$
Using identification of $U(\g)/I_\varphi$ with the Whittaker module $U(\g)\otimes_{U(\n)}\mathbb{C}_\varphi\simeq U(\b)\otimes \mathbb C$ we can consider $W$ as a subalgebra of $U(\b)$.
The natural projection $\vartheta: U(\b)\to U(\h)$ with the kernel
$\mathfrak{n}^+ U(\b)$ is called the {\it Harish-Chandra homomorphism}.
It is proven in \cite{PS2}
that the restriction of $\vartheta$ to $W$ is injective.

The center of $U(\g)$ is described in \cite{S}.
Set
$$\xi_i:=(-1)^{i+1}f_{i,i},\  x_i:=\xi_i^2=e_{i,i},$$ then $$U(\h)\simeq\mathbb C[\xi_1,\dots,\xi_n]/(\xi_i\xi_j+\xi_j\xi_i)_{i<j\leq n}.$$
The center of $U(\h)$ coincides with $\mathbb C[x_1,\dots,x_n]$ and the image of the center of $U(\g)$ under the Harish-Chandra homomorphism is generated
by the polynomials $p_{k}=x^{2k+1}_1+\dots+x_n^{2k+1}$ for all $k\in\mathbb N$. These polynomials are called $Q$-symmetric polynomials.

In \cite{PS2} we proved that the center $Z$ of $W$ coincides with the image of the center of $U(\g)$ and hence can be also identified with the ring of $Q$-symmetric polynomials.

\section {The structure of $W$-algebra}

Using Harish-Chandra
homomorphism we realize $W$ as a subalgebra
in $U(\h)$.
It is shown in \cite{PS3} that $W$ has $n$ even generators $z_0,\dots,z_{n-1}$ and $n$ odd generators  $\phi_0,\dots\phi_{n-1}$ defined as follows.
For $k\geq 0$ we set
 \begin{equation}\label{oddgen}
  \phi_0:=\sum_{i=1}^n\xi_i,\quad \phi_k:=T^{k}(\phi_0),
\end{equation}
where the matrix of $T$ in the standard basis $\xi_1,\dots,\xi_n$ has $0$ on the diagonal and
\begin{equation}\label{equmatrix}
t_{ij}:=\left\{ \begin{array}{cc}
 x_j &   \text{if} \quad  i<j,\\
 -x_j &  \text{if} \quad  i>j.\\
\end{array}\right.
\end{equation}
For odd $k\leq n-1$ we define
\begin{equation}\label{equA2}
z_{k}:= [\sum_{i_1< i_2\ldots < i_{k+1}}
(x_{i_1} + (-1)^{k}\xi_{i_1}) \cdots (x_{i_{k}} - \xi_{i_{k}})(x_{i_{k+1}} + \xi_{i_{k+1}})]_{even},
\end{equation}
and for even $k\geq 0$ we set
\begin{equation}\label{centgen}
  z_k:={1\over 2}[\phi_0,\phi_k].
\end{equation}

Let $W_0\subset W$ be the subalgebra generated by $z_0,\dots,z_{n-1}$. By \cite{PS2} Proposition 6.4, $W_0$ is isomorphic to the polynomial algebra
$\mathbb C[z_0,\dots,z_{n-1}]$.
Furthermore there are the following relations
\begin{equation}\label{oddrel}
  [\phi_i,\phi_j]=\begin{cases} (-1)^i 2z_{i+j}\,\text{if}\,\,i+j\,\text{is even}\\0 \,\text{if}\,\,i+j\,\text{is odd}\end{cases}
  \end{equation}

  Define the $\mathbb Z$-grading on $U(\h)$ by setting the degree of $\xi_i$ to be $1$. It induces the filtration on $W$, for every $y\in W$ we denote by $\bar y$
  the term of the highest degree.

  Note that for even $k$, we have $z_k=\bar z_k$.  Moreover, $z_k$ is
  in the image under the Harish-Chandra map of the center of the
  universal enveloping algebra $U(Q(n))$. Therefore by \cite{S}
  $z_{2p}$ is a $Q$-symmetric polynomial in $\mathbb C[x_1,\dots,x_n]$
  of degree $2p+1$. For example,
  $$z_0=x_1+\dots+x_n,\quad z_2=\frac{1}{3}\left((x_1^3+\dots+x_n^3) - (x_1+\dots+x_n)^3\right).$$
  For odd $k$ the leading term is given by the elementary symmetric polynomial
  $$\bar z_k=\sum_{i_1< i_2<\ldots <i_{k+1}}x_{i_1}\cdots x_{i_{k+1}}.$$

\begin{lemma}\label{structure}
  \begin{enumerate}
  \item $\operatorname{gr}W_0$ is isomorphic to the algebra of symmetric polynomials $\mathbb C[x_1,\dots,x_n]^{S_n}=\mathbb C[\bar z_0,\dots,\bar z_{n-1}]$
    and the degree of $\bar z_k$ is $2k+2$;
    \item $U(\h)$ is a free right $W_0$-module of rank $2^nn!$.
    \end{enumerate}
  \end{lemma}
  \begin{proof} Since $\bar z_0,\dots,\bar z_{n-1}$ are
    algebraically independent generators of $\mathbb C[x_1,\dots,x_n]^{S_n}$ we obtain (1).

    It is well-known fact that $\mathbb C[x_1,\dots,x_n]$ is a free $\mathbb C[x_1,\dots,x_n]^{S_n}$-module of rank $n!$, see, for example, \cite{Spr} Chapter 4.
    Since $U(\h)$ is a free
    $\mathbb C[x_1,\dots,x_n]$-module of rank $2^n$ we get that $U(\h)$ is a free $\mathbb C[x_1,\dots,x_n]^{S_n}$-module of rank $m=2^nn!$.
    Let us choose a homogeneous basis $b_1,\dots,b_m$ of $U(\h)$ over
    $\mathbb C[x_1,\dots,x_n]^{S_n}$. We claim that it is a basis of
    $U(\h)$
    as a right module over
    $W_0$. Indeed, let us prove first the linear independence. Suppose
    $$\sum_{j=1}^m b_jy_j=0$$
    for some $y_j\in W_0$. Let $k=\max\{\deg y_j+\deg b_j\,|j=1,\dots,m\}$. If $J=\{j\,|\,\deg y_j+\deg b_j=k\}$ we have $\sum_{j\in J}\bar b_jy_j=0$. By above this implies $\bar y_j=0$
    for all $j\in J$ and we obtain all $y_j=0$.
    On the other hand, it follows easily  by induction on degree that $U(\h)=\sum_{j=1}^mb_jW_0 $. The proof of (2) is complete.
  \end{proof}
 Consider $U(\h)$ as a free $U(\h_{\bar 0})$-module and let $W_1$ denote the free $U(\h_{\bar 0})$-submodule generated by $\xi_1,\dots,\xi_n$.
Then $W_1$ is equipped with $U(\h_{\bar 0})$-valued symmetric bilinear form $B(x,y)=[x,y]$.

   \begin{lemma}\label{cliff} Let $p(x_1,\dots,x_n):=\prod_{i<j}(x_i+x_j)$ and $\Gamma$ denotes the Gram matrix $B(\phi_i,\phi_j)$. Then
     $\operatorname{det}\Gamma=c p^2x_1\cdots x_n$, where $c$ is a non-zero constant.
   \end{lemma}
   \begin{proof} Recall that $\phi_k=T^{k}\phi_0$.
   Since the matrix
     of the form $B$ in the basis $\xi_1,\dots,\xi_n$ is the diagonal
     matrix
     $C=\operatorname{diag}(x_1,\dots,x_n)$, then $\Gamma=Y^t C Y$, where $Y$ is the square matrix such that $\phi_i=\sum_{j=1}^n y_{ij}\xi_j$.
     Hence det $\Gamma=x_1\cdots x_n\operatorname{det}Y^2$. Since $B(\phi_i,\phi_j)$ is a symmetric polynomial in $x_1,\dots,x_n$, the determinant of $\Gamma$ is also a
     symmetric polynomial.
     The degree of this polynomial is $n^2$. Therefore it suffices to prove that $(x_1+x_2)^2$ divides $\operatorname{det}\Gamma$,  or equivalently
     $x_1+x_2$ divides $\operatorname{det}Y$. In other words, we have to show that if $x_1=-x_2$, then $\phi_0,\dots,\phi_{n-1}$ are linearly dependent.
     Indeed, one can easily see from the form of $T$ that the first and the second coordinates of $T^k\phi_0$ coincide, hence $\phi_0,T\phi_0,\dots, T^{n-1}\phi_0$
     are linearly dependent.
   \end{proof}

  We also will use another generators in $W$ introduced in \cite{PS3}, Corollary 5.15:
           \begin{align}\label{gen1}
&u_{k}(0):= [\sum_{1\leq i_1< i_2<\ldots < i_k\leq n}
(x_{i_1} + (-1)^{k+1}\xi_{i_1}) \cdots (x_{i_{k-1}} - \xi_{i_{k-1}})(x_{i_{k}} + \xi_{i_{k}})]_{even},\\
&u_k(1):= [\sum_{1\leq i_1<i_2<\ldots <i_k\leq n}
(x_{i_1} + (-1)^{k+1}\xi_{i_1}) \cdots (x_{i_{k-1}} - \xi_{i_{k-1}})(x_{i_{k}} + \xi_{i_{k}})]_{odd}.
\nonumber
           \end{align}
           For convenience we assume $u_k(0)=u_k(1)=0$ for $k>n$.

           Let $i+j=n$. We have the natural embedding of the Lie superalgebras $Q(i)\oplus Q(j)\hookrightarrow Q(n)$. If $\h_r$ denotes the Cartan subalgebra of $Q(r)$,
           the above embedding induces the isomorphism
           \begin{equation}\label{tp}
             U(\h)\simeq U(\h_i)\otimes U(\h_j).
             \end{equation}
           The following lemma implies that we have also the embedding of $W$-algebras.
\begin{lemma}\label{insert} Let $i+j=n$. Then $W$ is a subalgebra in the tensor product
         $W^{i}\otimes W^{j}$, where $W^{r}\subset U(\h_r)$ denotes the $W$-algebra for
           $Q(r)$.
         \end{lemma}
         \begin{proof}
           Introduce generators in $W^{i}$ and $W^{j}$:
                      \begin{align}\label{gen1+}
&u^+_{k}(0):= [\sum_{1\leq i_1< i_2<\ldots < i_k\leq i}
(x_{i_1} + (-1)^{k+1}\xi_{i_1}) \cdots (x_{i_{k-1}} - \xi_{i_{k-1}})(x_{i_{k}} + \xi_{i_{k}})]_{even},\\
&u^+_k(1):= [\sum_{1\leq i_1<i_2<\ldots <i_k\leq i}
(x_{i_1} + (-1)^{k+1}\xi_{i_1}) \cdots (x_{i_{k-1}} - \xi_{i_{k-1}})(x_{i_{k}} + \xi_{i_{k}})]_{odd}.
\nonumber
           \end{align}
                      \begin{align}\label{gen1-}
&u^-_{k}(0):= [\sum_{i+1\leq i_1< i_2<\ldots < i_k\leq n}
(x_{i_1} + (-1)^{k+1}\xi_{i_1}) \cdots (x_{i_{k-1}} - \xi_{i_{k-1}})(x_{i_{k}} + \xi_{i_{k}})]_{even},\\
&u^-_k(1):= [\sum_{i+1\leq i_1<i_2<\ldots <i_k\leq n}
(x_{i_1} + (-1)^{k+1}\xi_{i_1}) \cdots (x_{i_{k-1}} - \xi_{i_{k-1}})(x_{i_{k}} + \xi_{i_{k}})]_{odd}.
\nonumber
                      \end{align}

 Then for $d,e,f\in \mathbb Z/2\mathbb Z$ we have
\begin{equation}\label{decomposition}
  u_k(d)=\sum_{e+f=d}\sum_{a+b=k}(-1)^{eb}u^+_a(e)u_b^-(f).
  \end{equation}
Here we assume $u_0^{\pm}(0)=1$ and $u_0^{\pm}(1)=0$.
\end{proof}
         \begin{corollary}\label{cor:insert} If $i_1+\dots+i_p=n$, then $W$ is a subalgebra in $W^{i_1}\otimes\dots\otimes W^{i_p}$.
           \end{corollary}
\section {Irreducible representations of $W$}

\subsection{Representations of $U(\h)$}
Let $\mathbf s=(s_1,\dots,s_n)\in\mathbb C^n$. We call $\mathbf s$ {\it regular} if $s_i\neq 0$ for all $i\leq n$ and {\it typical} if $s_i+s_j\neq 0$ for all
$i\not= j\leq n$.

It follows from the representation theory of Clifford algebras that all irreducible representations of $U(\h)$  up to
change of parity can be parameterized  by $\mathbf s\in\mathbb C^n$.
Indeed, let $M$ be an irreducible representation of $U(\h)$. By Schur's lemma every $x_i$ acts on $M$ as a scalar operator $s_i\operatorname{Id}$.
Let $I_{\mathbf s}$ denote the ideal in $U(\h)$ generated by
$x_i-s_i$, then the quotient algebra $U(\h)/I_{\mathbf s}$ is
isomorphic to the Clifford superalgebra $C_{\mathbf s}$
\footnote {We consider Clifford algebras as superalgebras with the natural $\mathbb Z_2$-grading} associated with the quadratic form:
$$B_{\mathbf s}(\xi_i,\xi_j)=\delta_{ij}2s_i.$$
Then $M$ is a simple $C_{\mathbf s}$-module.

The radical $R_{\mathbf s}$ of $C_{\mathbf s}$ is generated by the kernel of the form $B_{\mathbf s}$.
Let $m({\mathbf s})$ be the number of non-zero coordinates of
${\mathbf s}$, then $C_{\mathbf s}/R_{\mathbf s}$ is isomorphic to the
matrix
superalgebra $M(2^{\frac{m}{2}-1}|2^{\frac{m}{2}-1})$ for
even $m$ and to the superalgebra $M(2^{\frac{m-1}{2}})\otimes\mathbb C[\epsilon]/(\epsilon^2-1)$ for odd $m$.

Therefore $C_{\mathbf s}$ has one (up to isomorphism) simple $\Z_2$-graded module  $V(\mathbf s)$ of type
Q for odd $m({\mathbf s})$,
 and two simple modules $V(\mathbf s)$ and $\Pi V(\mathbf s)$ of type M for even $m({\mathbf s})$ (see \cite{M}).
 In the case when $\mathbf s$ is regular, the form $B_{\mathbf s}$ is non-degenerate
 and the dimension of $V(\mathbf s)$ equals $2^k$, where $k=\lceil n/2\rceil$. In general, $\dim V(\mathbf s)=2^{\lceil m(\mathbf s)/2\rceil}$.

 Consider the embedding $Q(p)\oplus Q(q)\hookrightarrow Q(n)$ for
 $p+q=n$ and the isomorphism (\ref{tp}). It induces an isomorphism of $U(\h)$-modules
 \begin{equation}\label{tensprod}
   V(\mathbf s)\simeq V(s_1,\dots,s_p)\boxtimes V(s_{p+1},\dots,s_n).
   \end{equation}

   \subsection{Restriction from $U(\h)$ to $W$}
 We denote by the same symbol $V(\mathbf s)$ the restriction to $W$ of the $U(\h)$-module $V(\mathbf s)$.
 \begin{proposition}\label{subquotient} Let $S$ be a simple
 $W$-module. Then $S$ is a simple constituent of $V(\mathbf s)$ for
 some
 $\mathbf s\in\mathbb C^n$.
 \end{proposition}
 \begin{proof} Since $W_0$ is commutative and $S$ is
 finite-dimensional (see \cite{PS2}), there exists one dimensional $W_0$-submodule
 $\mathbb C_{\nu}\subset S$ with character
   $\nu$. Therefore $S$ is a quotient of
 $\operatorname{Ind}^{W}_{W_0}\mathbb C_{\nu}$. On the other hand, the
 embedding $W\hookrightarrow U(\h)$ induces the
   embedding
   $\operatorname{Ind}^{W}_{W_0}\mathbb
 C_{\nu}\hookrightarrow\operatorname{Ind}^{U(\h)}_{W_0}\mathbb
 C_{\nu}$. Thus, $S$ is a simple constituent of
 $\operatorname{Res}_W\operatorname{Ind}^{U(\h)}_{W_0}\mathbb C_{\nu}$. By Lemma \ref{structure}, $\operatorname{Ind}^{U(\h)}_{W_0}\mathbb C_{\nu}$
   is finite-dimensional, and hence has simple $U(\h)$-constituents isomorphic to $V(\mathbf s)$ for some $\mathbf s$. Hence $S$ must appear as a
   simple $W$-constituent of some $V(\mathbf s)$.
 \end{proof}

  \subsection{Typical representations}

 \begin{theorem}\label{regular} If $\mathbf s$ is typical, then  $V(\mathbf s)$ is a simple $W$-module.
   \end{theorem}
   \begin{proof} First, we assume
     that $\mathbf   s$ is regular,
     i.e.  $s_i\neq 0$ for all $i=1,\dots,n$.
     The specialization  $x_i\mapsto s_i$ induces an injective
   homomorphism $\theta_{\mathbf s}: W/(I_{\mathbf s}\cap
   W)\hookrightarrow C_{\mathbf s}$
   and a specialization
of the quadratic form $B\mapsto B_{\mathbf s}$. By Lemma \ref{cliff}
     $\operatorname{det}\Gamma(\mathbf s)\neq 0$. Therefore $B_{\mathbf s}$ is non-degenerate and $\theta_{\mathbf s}$ is an isomorphism.
Thus, $V(\mathbf s)$ remains irreducible when restricted to $W$.

If $\mathbf s$ is typical non-regular, there is exactly one $i$ such
that $s_i=0$. Let $\mathbf s'=(s_1,\dots,s_{i-1},s_{i+1},\dots s_n)$.
Note that $(\theta_{\mathbf s}(\xi_i))$ is a nilpotent ideal  of $C_{\mathbf s}$ and hence
$\xi_i$ acts by zero on $V(\mathbf s)$.
Then $V(\mathbf s)$ is a simple module over the
quotient $C_{\mathbf s'}\simeq C_{\mathbf s}/(\theta_{\mathbf s}(\xi_i))$.
Recall
$Y$ from the proof of Lemma \ref{cliff} and let $Y'$
denote the minor of $Y$ obtained by removing the $i$-th column and the $i$-th row. Then
$$\phi_k=\sum_{j\neq i}y'_{kj}\xi_j \mathrm{mod}\,(\xi_i).$$
Hence $\theta_{\mathbf s}(\phi_0),\dots,\theta_{\mathbf s}(\phi_{n-1})$
generate
$C_{\mathbf s'}\simeq C_{\mathbf s}/(\theta_{\mathbf s}(\xi_i))$
and the statement follows from the regular
case for $n-1$.
\end{proof}

\subsection{Simple $W$-modules for $n=2$}
Let $n=2$, then by Theorem \ref{regular} $V(\mathbf s)$ is simple as $W$-module if $s_1\neq -s_2$.  The action of $U(\h)$ in $V(s_1,s_2)$ is given by the following formulas in a
   suitable basis:
   $$\xi_1\mapsto \left(\begin{matrix}0&\sqrt{s_1}\\ \sqrt{s_1}&0\end{matrix}\right),\quad \xi_2\mapsto \left(\begin{matrix}0&\sqrt{s_2}{\mathbf i}\\ -\sqrt{s_2}{\mathbf i}&0\end{matrix}\right).$$
   Note that $W$ is generated by $\phi_0, \phi_1, z_0$ and $z'_1$, where $z'_1 := u_2(0)$.
  Using
   $$\phi_0=\xi_1+\xi_2,\,\,\phi_1=x_2\xi_1-x_1\xi_2,\,\,z_0=x_1+x_2,\,\,z'_1=x_1x_2-\xi_1\xi_2$$
   we obtain the following formulas for the generators of $W$:
\begin{equation}\label{formula1}
         \phi_0\mapsto\left(\begin{matrix}0&\sqrt{s_1}+\sqrt{s_2}{\mathbf i}\\ \sqrt{s_1}-\sqrt{s_2}{\mathbf i}&0\end{matrix}\right),\quad
         \phi_1\mapsto\sqrt{s_1s_2}\left(\begin{matrix}0&\sqrt{s_2}-\sqrt{s_1}{\mathbf i}\\ \sqrt{s_2}+\sqrt{s_1}{\mathbf i}&0\end{matrix}\right),
         \end{equation}
\begin{equation}\label{formula2}
  z_0\mapsto (s_1+s_2)\hbox{Id},\quad z'_1\mapsto \left(\begin{matrix}s_1s_2+\sqrt{s_1s_2}{\mathbf i}&0\\ 0&s_1s_2-\sqrt{s_1s_2}{\mathbf i}\end{matrix}\right).
\end{equation}
 Assume that $s_1=-s_2$. If  $s_1,s_2=0$ then $V(\mathbf s)$ is
isomorphic to $\mathbb C\oplus \Pi\mathbb C$, where $\mathbb C$ is the trivial module. If $s_1\neq 0$, we choose
$\sqrt{s_1},\sqrt{s_2}$ so that
$\sqrt{s_2}=\sqrt{s_1}\mathbf{i}$. Note that the choice of sign controls the choice
of the parity of $V(\mathbf s)$.
The following  exact sequence
easily follows from (\ref{formula1}) and (\ref{formula2}):
\begin{equation}\label{exact1}
0\to \Pi\Gamma_{-s^2_1+s_1}\to V(\mathbf s)\to \Gamma_{-s^2_1-s_1}\to 0 ,
\end{equation}
where $\Gamma_t$ is the simple module of dimension $(1|0)$ on which $\phi_0,\phi_1$ and $z_0$ act by zero and $z'_1$ acts by the scalar $t$.
The sequence splits only in the case $s_1=0$, when $\Gamma_0\simeq\mathbb C$ is trivial.
Thus, using Proposition \ref{subquotient}, Theorem \ref{regular}, and (\ref{exact1}) we obtain

\begin{lemma}\label{n=2} If $n=2$, then every simple $W$-module is isomorphic to one of the following
  \begin{enumerate}
  \item $V(s_1,s_2)$ or $\Pi V(s_1,s_2)$ for $s_1\neq -s_2$, $s_1,s_2\neq 0$;
  \item $V(s,0)$ if $s\neq 0$;
    \item $\Gamma_t$ or $\Pi\Gamma_t$.
\end{enumerate}
\end{lemma}

\subsection{Invariance under permutations}

\begin{theorem} \label{symmetry} Let $\mathbf s'=\sigma (\mathbf s)$ for
   some permutation of coordinates.

   \noindent
   (1) If $\mathbf s$ is typical, then $V(\mathbf s)$ is isomorphic to $V(\mathbf s')$ as a $W$-module.

   \noindent
   (2) If $\mathbf s$ is arbitrary, then $[V(\mathbf s)]=[V(\mathbf s')]$ or $[\Pi V(\mathbf s')]$, where $[X]$ denotes the class of $X$ in the Grothendieck group.

   \noindent

       \end{theorem}
       \begin{proof} First, we will prove the statement for $n=2$.
  Assume first that $s_2\neq -s_1$.
  In this case $V(s_1,s_2)$ is a $(1|1)$-dimensional simple $W$-module.

Let
         $$D=\left(\begin{matrix}\sqrt{s_2}+\sqrt{s_1}{\mathbf i}&0\\
         0&\sqrt{s_1}+\sqrt{s_2}{\mathbf i}\end{matrix}\right).$$
         Then by direct computation we have
         $$D\phi_0D^{-1}=\left(\begin{matrix}0&\sqrt{s_2}+\sqrt{s_1}{\mathbf
         i}\\ \sqrt{s_2}-\sqrt{s_1}{\mathbf i}&0\end{matrix}\right)$$
         and
        $$D\phi_1D^{-1}=\sqrt{s_1s_2}\left(\begin{matrix}0&\sqrt{s_1}-\sqrt{s_2}{\mathbf     i}\\ \sqrt{s_1}+\sqrt{s_2}{\mathbf i}&0\end{matrix}\right).$$
        Therefore $D$ defines an isomorphism between $V(s_1,s_2)$ and $V(s_2,s_1)$.

        Now consider the case $s_1=-s_2$. Then the structure of $V(s_1,-s_1)$ is given by the sequence (\ref{exact1}). Let $V(\mathbf s') = V(-s_1, s_1)$,
        then analogously we have
the exact sequence
\begin{equation}\label{exact2}
0\to \Pi\Gamma_{-s^2_1-s_1}\to V(\mathbf s')\to \Gamma_{-s^2_1+s_1}\to 0 .
 \end{equation}
The statement (2) now follows directly from comparison of (\ref{exact1}) and (\ref{exact2}).
         Now we will prove the statement for all $n$. Note that it
         suffices to consider the case of the adjacent
         transposition $\sigma=(i,i+1)$.

         The embedding of $Q(i-1)\oplus Q(2)\oplus Q(n-i-1)$ into $Q(n)$ provides the isomorphism
         $$U(\h)\simeq U(\h^-)\otimes U(\h^0)\otimes U(\h^+),$$
         where $\h^-$, $\h^0$ and $\h^+$ are the Cartan subalgebras of
         $Q(i-1), Q(2)$ and $Q(n-i-1)$ respectively.
Using twice the isomorphism (\ref{tensprod}) we obtain the following isomorphism of $U(\h)$-modules
         $$ V(\mathbf s)\simeq (V(s_1,\dots,s_{i-1})\boxtimes V(s_i,s_{i+1}))\boxtimes V(s_{i+2},\dots,s_n).$$

         Suppose that $s_i\not=-s_{i+1}$.
Let $D_{i,i+1}=1\otimes D\otimes 1$. By Corollary \ref{cor:insert} we have that $W$ is a subalgebra in $W^{i-1}\otimes W^2\otimes W^{n-i-1}$ and hence $D_{i,i+1}$
defines an isomorphism of $W$-modules  $V(\mathbf s)$ and $V(\mathbf s')$.

If $s_i =-s_{i+1}$, then the statement follows from  (\ref{exact1})
and (\ref{exact2}).
This completes the proof  of the theorem.

      \end{proof}

\subsection{Construction of simple $W$-modules}
Now we give a general construction of a simple $W$-module.
Let $r, p,q\in\mathbb N$ and $r+2p+q=n$, $\mathbf t=(t_1,\dots,t_p)\in\mathbb C^p$, $t_1,\dots, t_p\neq 0$,
and $\lambda=(\lambda_1,\dots,\lambda_q)\in\mathbb C^q$,
$\lambda_1,\dots, \lambda_q\neq 0$, such that
$\lambda_i+\lambda_j\neq 0$ for any $1\leq i\not= j\leq q$.
Recall that by Corollary \ref{cor:insert} we have an embedding $W\hookrightarrow W^r\otimes (W^2)^{\otimes p}\otimes W^q$. Set
$$S(\mathbf t,\lambda):=\mathbb C\boxtimes \Gamma_{t_1}\boxtimes\dots\boxtimes\Gamma_{t_p}\boxtimes V(\lambda),$$
where the first term $\mathbb C$ in the tensor product denotes the trivial $W^r$-module.
For $q=0$ we use the notation $S(\mathbf t,0)$ and set $V(\lambda) = \mathbb C$.

\begin{remark} The dimension of $S(\mathbf t,\lambda)$ equals $2^{\frac{q}{2}}$ for even $q$ and  $2^{\frac{q+1}{2}}$ for odd $q$. Furthermore, $S(\mathbf t,\lambda)$
  is isomorphic to $\Pi S(\mathbf t,\lambda)$ if and only if $q$ is odd.
  \end{remark}
\begin{lemma}\label{one-dim} All $u_k(1)$ act by zero on $S(\mathbf t,0)$. The action of $u_k(0)$ is given by the formula
  $$u_k(0)=\begin{cases} 0\,\,\text{for odd}\,\,k,\,\,\text{and for}\,\,k>2p,\\ \sigma_{\frac{k}{2}}(t_1,\dots,t_p)\,\,\text{for even}\,\,k,
  \end{cases}$$
     where $\sigma_a$ denote the elementary symmetric polynomials, $0\leq a\leq p$.
\end{lemma}
\begin{proof} The first assertion is trivial. We prove the second assertion by induction on $p$. For $p=1$ it is a consequence of the definition of $\Gamma_t$
  for $Q(2)$. For $p>1$ we consider the embedding $Q(n-2)\oplus Q(2)\hookrightarrow Q(n)$.
 The formula (\ref{decomposition}) degenerates to
  $$u_k(0)=u_k^+(0)\otimes 1+u_{k-1}^+(0)\otimes z_0+u_{k-2}^+(0)\otimes z'_1.$$
  As $z_0$ acts by zero on $\Gamma_{t_p}$ the statement now follows from the obvious identity
  $$\sigma_{\frac{k}{2}}(t_1,\dots t_p)=\sigma_{\frac{k}{2}}(t_1,\dots t_{p-1})+t_p\sigma_{\frac{k}{2}-1}(t_1,\dots t_{p-1}).$$
  \end{proof}

\begin{theorem}\label{irreducible}
  \begin{enumerate}
  \item $S(\mathbf t,\lambda)$ is a simple $W$-module;
  \item Every simple $W$-module is isomorphic to
  $S(\mathbf t,\lambda)$ up to change of parity.
\end{enumerate}
\begin{proof} Let $u^-_k(d)$, $d\in\mathbb Z/2\mathbb Z,\,1\leq k\leq n$ be as in (\ref{gen1-}) where indices are taken in the
  interval $[n-q+1,n]$. If $q=0$ we set $u^-_k(0)=1$ and $u^-_k(1)=0$.
Using Lemma \ref{one-dim} and formula (\ref{decomposition})
  we can easily write the action of  $u_k(d)$ in
  $S(\mathbf t,\lambda)$ in terms of $u^-_k(d)$ after
  identifying $S(\mathbf t,\lambda)$ with $V(\lambda)$:
  \begin{equation}\label{simpleaction}
  u_k(d)=\sum_{2a+j=k} \sigma_{a}(t_1,\dots,t_p)u_j^-(d),
  \end{equation}

 From these
  formulas we see that $u^-_k(d)$ and $u_k(d)$ generate the same
  subalgebra in $\operatorname{End}_{\mathbb C}(V(\lambda))$. By
  Theorem \ref{regular} this proves irreducibility of $S(\mathbf t,\lambda)$.

To show (2) we use Proposition \ref{subquotient}. Every simple
$W$-module is a subquotient of $V(\mathbf s)$. By Theorem \ref{symmetry} (2)
we may assume that $s_1=\dots  = s_r=0$, $s_i\neq 0$ for $i>r$,
$s_{r+1}=-s_{r+2},\dots, s_{r+2p-1}=-s_{r+2p}$. We can compute
$W^r\otimes (W^2)^{\otimes p}\otimes W^q$-simple
constituents of $V(\mathbf s)$. They are
$S(\mathbf t,\lambda)$ (up to change of parity) with $t_j=-s_{r+2j}^2\pm s_{r+2j}$ and $\lambda_i=s_{r+2p+i}$
(we can assume that all $s_i\not= \pm 1$).
    By (1) $S(\mathbf t,\lambda)$ remains simple when
    restricted to $W$. Hence the statement.
    \end{proof}
  \end{theorem}
  \begin{remark}
  $\Gamma_{0}\simeq \mathbb C \boxtimes \mathbb C$ as $W^2$-modules ($r =2$, $p = q = 0$).
  \end{remark}

    \subsection{Central characters} Recall that the center of $U(Q(n))$ coincides with the center $Z$ of $W$, see Section 2.
    Every $\mathbf s$ defines the central character
    $\chi_{\mathbf s}:Z\to\mathbb C$. Furthermore, Theorem \ref{irreducible} (2) implies that every simple $W$-module admits central character $\chi_{\mathbf s}$ for
    some $\mathbf s$.
    For every $\mathbf s=(s_1,\dots,s_n)$ we define the {\it core} $c(\mathbf s)=(s_{i_1},\dots,s_{i_m})$ as a subsequence obtained
    from $\mathbf s$ by removing all $s_j=0$ and all pairs $(s_i,s_j)$ such that $s_i+s_j=0$. Up to a  permutation this result does not depend on the order of
    removing. Thus, the core is well defined up to permutation. We call $m$ the length of the core. The notion of core is very useful for describing the blocks in the
    category of finite-dimensional $Q(n)$-modules, see \cite{Pen} and \cite{Ser}.
    \begin{example} Let $\mathbf s=(1,0,3,-1,-1)$, then $c(\mathbf s)=(3,-1)$.
      \end{example}

    The following is a reformulation of the central character description in \cite{S}.
    \begin{lemma}\label{cent} Let $\mathbf s,\mathbf s'\in\mathbb C^n$. Then $\chi_{\mathbf s}=\chi_{\mathbf s'}$ if and only if
      $\mathbf s$ and $\mathbf s'$ have the same core (up to permutation).
    \end{lemma}
    It follows from Lemma \ref{cent} that the core depends only on the central character $\chi_{\mathbf s}$, we denote it $c(\chi)$. By Theorem \ref{symmetry}
    we obtain the following.
    \begin{corollary}\label{corerep} Let $\chi:Z\to\mathbb C$ be a central character with core $c(\chi)$ of length $m$. Then $W^m$-module
      $V(c(\chi))$ is well-defined.
      From now on we denote it by $V(\chi)$ and call it the {\it core representation}.
    \end{corollary}

    The category $W-\mathrm{mod}$ of finite dimensional $W$-modules
    decomposes into direct sum $\bigoplus W^\chi-\mathrm{mod}$, where $ W^\chi-\mathrm{mod}$ is the full subcategory of modules admitting generalized central
    character $\chi$.

    \begin{lemma}\label{simplcent} A simple $W$-module $S$ belongs to $ W^\chi-\mathrm{mod}$ if and only if it is isomorphic  (up to change of parity)
      to $S(\mathbf t,\lambda)$
      with $\lambda=c(\chi)$.
    \end{lemma}
    \begin{proof} We have to compute the central character of $S(\mathbf t,\lambda)$. For a $Q$-symmetric polynomial $p_k=x^{2k+1}_1+\dots+x^{2k+1}_n$ we have
      $p_k(\mathbf t,\lambda)=\lambda_1^{2k+1} + \dots + \lambda_q^{2k+1}$. Since $p_k$ generate the center of $W$ the statement follows.
      \end{proof}

      \begin{proposition}\label{classification} Two simple modules $S(\mathbf t,\lambda)$ and $S(\mathbf t',\lambda')$ are isomorphic (up to change of parity)
        if and only if $p'=p$, $q'=q$,
$\mathbf t'=\sigma(\mathbf t)$ and $\lambda'=\tau(\lambda)$ for some $\sigma\in S_p$ and $\tau\in S_q$.
      \end{proposition}
      \begin{proof}  First, (\ref{simpleaction}) and Theorem \ref{symmetry} imply the ``if'' statement. To prove the ``only if'' statement, assume that
$S(\mathbf t,\mathbf \lambda)$ and $S(\mathbf t',\lambda')$ are isomorphic. Then these modules admit the same central character. Therefore
by Lemma \ref{simplcent} $\lambda'=\tau(\lambda)$ for some $\tau\in S_q$. Hence without loss of generality we may assume that $q'=q$ and $\lambda'=\lambda$.

Denote by
$\operatorname{tr}x$ and $\operatorname{tr}'x$ the trace of $x\in W$ in $S(\mathbf t,\lambda)$ and $S(\mathbf t',\lambda)$ respectively.
Then we must have
$$\operatorname{tr}u_k(0)=\operatorname{tr}'u_k(0).$$
Using the formula (\ref{simpleaction}) we get
$$ \operatorname{tr}u_k(0)=\sum_{2a+j=k} \sigma_{a}(t_1,\dots,t_p) \operatorname{tr}_{V(\lambda)} u_j^-(0),$$
$$ \operatorname{tr}'u_k(0)=\sum_{2a+j=k} \sigma_{a}(t'_1,\dots,t_{p'}') \operatorname{tr}_{V(\lambda)} u_j^-(0).$$
Let $b_j:=\operatorname{tr}_{V(\lambda)} u_j^-(0)$. Without loss of generality we may assume that $p\geq p'$. Then we can rewrite our formula with $p=p'$ assuming
$t_i'=0$ for $p\geq i>p'$. Then the above implies
$$\sigma_{a}(t_1,\dots,t_p)b_0+\sigma_{a-1}(t_1,\dots,t_p)b_2+\dots+\sigma_{0}(t_1,\dots,t_p)b_{2a}=$$
$$\sigma_{a}(t'_1,\dots,t'_p)b_0+\sigma_{a-1}(t'_1,\dots,t'_p)b_2+\dots+\sigma_{0}(t'_1,\dots,t'_p)b_{2a},$$
where we assume $b_i=0$ for $i>q$. Since $b_0=\dim V(\lambda)\neq 0$ the above equations imply $\sigma_a(t_1,\dots,t_p)=\sigma_a(t'_1,\dots,t_p')$ for all $a=1,\dots,p$.
Therefore $\mathbf t'=\sigma(\mathbf t)$ for some $\sigma\in S_p$ and in particular, $p'=p$.
      \end{proof}
    We denote  by $\mathcal P^l$ the subcategory of $W^l$-modules which admit trivial generalized central character.
    \begin{lemma} \label{functor}Let $\chi:Z\to\mathbb C$ be a central character with core $c(\chi)$ of length $m$. Then the functor
      $W^{n-m}-\mathrm{mod}\to  W-\mathrm{mod}$ defined by $F(M)=\operatorname{Res}_W(M\otimes V(\chi))$
      \footnote{We consider here the usual exterior tensor product in contrast with $\boxtimes$}  restricts to the functor
      $\Phi:\mathcal P^{n-m} \to W^\chi-\mathrm{mod}$.  Furthermore, $\Phi$ is an exact functor which sends a simple object to a simple object.
      \end{lemma}
      \begin{proof} The first assertion is immediate consequence of
    Lemma \ref{simplcent} and the second follows from the
    construction of $S(\mathbf t,\lambda)$.
        \end{proof}
  \begin{conjecture}\label{equivalence} The functor $\Phi:\mathcal    P^{n-m} \to W^\chi-\mathrm{mod}$
    defines an equivalence of categories.
  \end{conjecture}

  \vskip 0.1 in

  \section{Representations of the super Yangian of type  $Q(1)$}

  \vskip 0.1in

  Recall that in \cite{N1} the Yangians $YQ(n)$ associated with Lie superalgebras $Q(n)$ were defined. In \cite{PS2} and \cite{PS3} (Corollary 5.16) we have shown
  the existence of the surjective homomorphism $\varphi_n:YQ(1)\to W^n$, where $W^n$ is the finite $W$-algebra associated with principal nilpotent orbit in
  $Q(n)^*$. In this section we classify irreducible finite-dimensional representations of $YQ(1)$ and explore its connections with irreducible
  representations of $W^n$.

Recall that
$YQ(1)$ is the associative unital superalgebra over $\C$ with the countable set of generators
$$T_{i,j}^{(m)}\hbox{ where }m = 1, 2, \ldots \hbox{ and }i, j = \pm 1.$$

\noindent
The $\Z_2$-grading of the algebra $YQ(1)$ is defined as follows:
$$p(T_{i,j}^{(m)})  =
p(i) + p(j), \hbox{ where } p(1) = 0 \hbox{ and }  p(-1) = 1.$$
To write down defining relations for these generators we employ the formal series

\noindent
in $YQ(1)[[u^{-1}]]$:
\begin{equation}\label{equY1}
T_{i,j}(u) = \delta_{ij}\cdot 1 + T_{i,j}^{(1)}u^{-1} + T_{i,j}^{(2)}u^{-2} + \ldots.
\end{equation}
Then for all possible indices $i, j, k, l$ we have the relations

\begin{align}\label{equY2}
& (u^2 - v^2)[T_{i,j}(u), T_{k,l}(v)]\cdot (-1)^{p(i)p(k) + p(i)p(l)  +  p(k)p(l)}\\
& = (u + v)(T_{k,j}(u)T_{i,l}(v) - T_{k,j}(v)T_{i,l}(u))\nonumber\\
&-(u - v)(T_{-k,j}(u)T_{-i,l}(v) - T_{k,-j}(v)T_{i,-l}(u))\cdot (-1)^{p(k) +p(l)},
\nonumber
\end{align}
 where $v$ is a formal parameter independent of $u$, so that (\ref{equY2}) is an equality in the algebra of formal Laurent series in $u^{-1}, v^{-1}$ with
coefficients in $YQ(1)$.

\noindent
 For all indices $i, j$ we also have the relations
 \begin{equation}\label{equY3}
 T_{i,j}(-u) = T_{-i,-j}(u).
 \end{equation}

\noindent
Note that the relations (\ref{equY2}) and (\ref{equY3}) are equivalent to the following defining relations:

 \begin{align}\label{equY4}
 & ([T_{i,j}^{(m+1)}, T_{k,l}^{(r-1)}] - [T_{i,j}^{(m-1)}, T_{k,l}^{(r+1)}])
 \cdot (-1)^{p(i)p(k) + p(i)p(l)  + p(k)p(l)}  = \\
 & T_{k,j}^{(m)}T_{i,l}^{(r-1)} + T_{k,j}^{(m-1)}T_{i,l}^{(r)} -
 T_{k,j}^{(r-1)}T_{i,l}^{(m)} - T_{k,j}^{(r)}T_{i,l}^{(m-1)}\nonumber\\
 & + (-1)^{p(k) + p(l)}(-T_{-k,j}^{(m)}T_{-i,l}^{(r-1)} + T_{-k,j}^{(m-1)}T_{-i,l}^{(r)} +
 T_{k,-j}^{(r-1)}T_{i,-l}^{(m)} - T_{k,-j}^{(r)}T_{i,-l}^{(m-1)}),
 \nonumber
\end{align}

 \begin{equation}\label{equY5}
 T_{-i,-j}^{(m)} = (-1)^m T_{i,j}^{(m)},
\end{equation}
where $m, r = 1, \ldots$ and $T_{i,j}^{(0)} = \delta_{ij}$.

Recall that $YQ(1)$ is a Hopf superalgebra, see \cite{N},  with comultiplication given by the formula
$$\Delta(T_{i,j}^{(r)})=\sum_{s=0}^r\sum_{k} (-1)^{(p(i)+p(k))(p(j)+p(k))}T_{i,k}^{(s)}\otimes T_{k,j}^{(r-s)}.$$

There exists a surjective homomorphism $\varphi_n:YQ(1)\to W^n$ defined as follows:
$$\varphi_n(T_{1,1}^{(k)})=(-1)^k [\sum_{1\leq i_1 < i_2 <\ldots < i_k\leq n}
(x_{i_1} + (-1)^{k+1}\xi_{i_1}) \ldots (x_{i_{k-1}} - \xi_{i_{k-1}})(x_{i_{k}} + \xi_{i_{k}})]_{even},$$
$$\varphi_n(T_{-1,1}^{(k)})=(-1)^k [\sum_{1\leq i_1 < i_2 <\ldots < i_k\leq n}
(x_{i_1} + (-1)^{k+1}\xi_{i_1}) \ldots (x_{i_{k-1}} - \xi_{i_{k-1}})(x_{i_{k}} + \xi_{i_{k}})]_{odd}.$$
Note that $\varphi_n(T_{1,1}^{(k)}) = \varphi_n(T_{-1,1}^{(k)}) = 0$ if $k > n$.

 \begin{lemma}\label{center} Let
   \begin{equation}\label{etadef}
     \eta_i=(-\frac{1}{2})^i\operatorname{ad}^iT_{1,1}^{(2)}(T_{1,-1}^{(1)}),\quad Z_{2i}=\frac{1}{2}[\eta_0,\eta_{2i}].
     \end{equation}
   \begin{enumerate}
   \item The relation (\ref{oddrel}) holds.
   \item The elements $\{Z_{2i}\mid i\in\mathbb N\}$ are algebraically independent generators of the center of $YQ(1)$.
   \item The elements $\eta_0$ and $\{T_{1,1}^{(2i)} \mid i\in\mathbb N\}$ generate $YQ(1)$.
     \end{enumerate}
   \end{lemma}
    \begin{remark}\label{correspondence}
We have the following correspondence between generators of $YQ(1)$ and $W^n$
 $$\varphi_n(\eta_i) = \phi_i, \quad \varphi_n(Z_{2i}) = z_{2i}, \quad 0\leq i\leq n-1$$
\end{remark}
   \begin{proof} It follows from the similar statements for $W^n$ for all $n$ and the fact that $\bigcap_{n\in\mathbb N}\operatorname{Ker}\varphi_n=0$.
   \end{proof}

   It is easy to see the following commutative diagram:
   \begin{equation}\label{diagram}
\begin{CD}
YQ(1) @>\Delta>> YQ(1)\otimes YQ(1) \\
@V\varphi_{m+n}VV @V\varphi_m\otimes\varphi_nVV \\
W^{m+n} @>>> W^m\otimes W^n
\end{CD}
\end{equation}
where the bottom horizontal arrow is the composition of the flip $W^n\otimes W^m\to W^m\otimes W^n$ with the map $W^{m+n}\to W^n\otimes W^m$ defined in Lemma \ref{insert}.
The appearance of the flip is due to the fact that the flip is used in the identification of $U(\mathfrak h)\subset U(Q(l))$ with $U(Q(1))^{\otimes l}$,
see the formula before Theorem 5.8 and Theorem 5.14 in \cite{PS3}.

   Let $M$ be a simple $YQ(1)$-module. Then $M$ admits the central character $\chi$. We set $\chi_{2k}=\chi(Z_{2k})$ and consider the generating function
   $$\chi(u)= \sum_{i=0}^\infty \chi_{2i}u^{-2i-1}.$$

   \begin{lemma}\label{necessary} Let $M$ be a finite-dimensional simple $YQ(1)$-module admitting central character $\chi$. Then $\chi(u)$ is a rational function
     of the form $\frac{a_0u^{-1}+\dots+a_pu^{-2p-1}}{1+c_1u^{-2}+\dots+c_qu^{-2q}}$.
   \end{lemma}
   \begin{proof} Let $\mathbf{C}\subset YQ(1)$ denote the unital subalgebra generated by $\{\eta_i\mid i\in\mathbb N\}$. Let $\mathbf{C}_{\chi}$ denote the quotient of
     $\mathbf C$ by the ideal $(\{Z_{2i}-\chi_{2i}\mid i\in\mathbb N\})$. Then the relations (\ref{oddrel}) imply that $\mathbf{C}_{\chi}$ is isomorphic to the
     infinite-dimensional Clifford algebra $\mathbf{Cliff}(V, B_{\chi})$ on the space $V$ with basis $\{\eta_i\mid i\in\mathbb N\}$ and the symmetric form
     $B_{\chi}$ defined by the formula
     \begin{equation}\label{form}
  B_{\chi}(\eta_i,\eta_j)=\begin{cases} (-1)^i 2\chi_{i+j}\,\text{if}\,\,i+j\,\text{is even}\\0 \,\text{if}\,\,i+j\,\text{is odd}\end{cases}.
  \end{equation}
  Note that $M$ by definition restricts to a certain $\mathbf{C}_\chi$-module. On the other hand, $\mathbf{Cliff}(V, B_{\chi})$ admits a
  finite-dimensional representation if and only if $B_{\chi}$ has a finite rank. Look at the infinite symmetric matrix of $B_{\chi}$ in the basis $\{\eta_i\}$.
  Then every column of this matrix is a linear combination of the first $k$ columns for some $k$. The formula (\ref{form}) implies that for some positive integers
  $q$ and $s$ and the coefficients $c_1,\dots,c_q$ we have a recurrence relation
\begin{equation}\label{recursion}
  \chi_{2m}=\sum^q_{i=1}-c_i\chi_{2m-2i},\quad\text{for all}\ m\geq s.
  \end{equation}
  This condition is equivalent to the rationality of $\chi(u)$.
\end{proof}

Recall the $W^n$-module  $V(\mathbf s)$ constructed in Section 4. Using the homomorphism $\varphi_n$ we equip $V(\mathbf s)$ with a $YQ(1)$-module structure.
Our next goal is to compute the
central character of $V(\mathbf s)$. For this we need to compute the $\{z_{2i}\}$ in terms of symmetric polynomials. Recall the notations of Section 3.
Note that for any $n$ the elements $\{z_{2i}\}$ of the center can be expressed in terms of symmetric polynomials of $x_1,\dots,x_n$ and this expression stabilizes
as $n\to\infty$. Thus, $z_{2i}$ is a particular element in the ring of symmetric functions of degree $2i+1$.
\begin{lemma}\label{description} We have the following expression
\begin{equation}\label{mainexpression}
  z_{2k}=-\sum_{i=1}^{k}\sigma_{2i}z_{2k-2i}+\sigma_{2k+1},
  \end{equation}
  where $\sigma_p=\sum_{i_1<\dots<i_p}x_{i_1}\dots x_{i_p}$ is the elementary symmetric function.
  \end{lemma}
  \begin{proof} We proved in \cite{PS2}, Lemma 5.5 that for $W^n$
  the characteristic polynomial $\operatorname{det} (\lambda\operatorname{Id}-T)$ of $T$ equals
  $\lambda^n+\sum_{i=1}^{\lfloor n/2\rfloor}\sigma_{2i}\lambda^{n-2i}$. As
    $$z_{2k}(x_1,\dots,x_n)=[x_1,\dots,x_n]T^{2k}[1,\dots,1]^t,$$
    the Hamilton--Cayley identity implies that for $2k\geq n$ we have
    $$z_{2k}(x_1,\dots,x_n)=-\sum_{i=1}^{k}\sigma_{2i}z_{2k-2i}(x_1,\dots,x_n).$$
    Since the degree of $z_{2k}$ is $2k+1$ it is a polynomial of $\sigma_1,\dots,\sigma_{2k+1}$. Therefore it suffices to prove (\ref{mainexpression}) for $n=2k+1$. We do it
    by induction on $k$ using the fact that
    $z_{2k}(x_1,\dots,x_{2k+1})$ is $Q$-symmetric. Indeed, we already know that
    $$z_{2k}(x_1,\dots,x_{2k})=-\sum_{i=1}^{k}\sigma_{2i}z_{2k-2i}(x_1,\dots,x_{2k}),$$
    therefore from substituting $x_{2k+1}=0$ we get
    $$z_{2k}(x_1,\dots,x_{2k+1})=-\sum_{i=1}^{k}\sigma_{2i}z_{2k-2i}(x_1,\dots,x_{2k+1})+A\sigma_{2k+1}(x_1,\dots,x_{2k+1}).$$
    It remains to find the coefficient $A$. By $Q$-symmetry
    $$z_{2k}(x_1,\dots,x_{2k-1})=z_{2k}(x_1,\dots,x_{2k-1},t,-t).$$
    This leads to the identity
    $$z_{2k}(x_1,\dots,x_{2k-1})=-\sum_{i=1}^{k}\sigma_{2i}z_{2k-2i}(x_1,\dots,x_{2k-1})+$$
    $$+t^2\sum_{i=1}^{k}\sigma_{2i-2}z_{2k-2i}(x_1,\dots,x_{2k-1})-
    At^2\sigma_{2k-1}(x_1,\dots,x_{2k-1}).$$
    Furthermore, by induction assumption we have
    $$\sum_{i=1}^{k}\sigma_{2i-2}z_{2k-2i}(x_1,\dots,x_{2k-1})=$$
    $$z_{2k-2}(x_1,\dots,x_{2k-1})+\sum_{i=2}^{k}\sigma_{2i-2}z_{2k-2i}(x_1,\dots,x_{2k-1})=
    \sigma_{2k-1}(x_1,\dots,x_{2k-1})$$
    Hence $A=1$.
  \end{proof}

     \begin{corollary}\label{character} $YQ(1)$-module $V(\mathbf s)$
 admits central character $\chi$ where
 $$\chi(u) = \frac{\sum_{i=0}^{\infty}\sigma_{2i+1}(\mathbf s)u^{-2i-1}} {1+\sum_{i=1}^{\infty}\sigma_{2i}(\mathbf s)u^{-2i}}.$$
\end{corollary}

  \begin{corollary}\label{independence} The elements $\{z_{2k},\mid k=0,\dots,\lfloor\frac{n-1}{2}\rfloor\}$ and
  $\{\sigma_{2k},\mid k=1,\dots,\lfloor\frac{n}{2}\rfloor\}$ form an algebraically independent set of generators in the ring of symmetric polynomials in $n$ variables.
    \end{corollary}

  \begin{proposition}\label{existence} For any rational $\chi(u)$ there exist $n$ and $\mathbf s$ such that $V({\mathbf s})$ admits central character $\chi$.
  \end{proposition}
  \begin{proof} It follows immediately from Corollary \ref{character}. Indeed, by Lemma \ref{necessary}
    $$\chi(u)=\frac{a_1u^{-1}+\dots+a_pu^{-2p-1}}{1+c_1u^{-2}+\dots+c_qu^{-2q}}.$$
    Let $n=\max(2p+1,2q)$ and assume that $a_i=0$ for $i>p$, $c_j=0$ for $j>q$. One can choose
    $\mathbf s=(s_1,\dots,s_n)$ so that
    $\sigma_{2k}(s_1,\dots,s_n)=c_k$ and $\sigma_{2k+1}(s_1,\dots,s_n)=a_k$.
     \end{proof}
  \begin{corollary}\label{bigclifford} Any simple finite-dimensional $\mathbf C$-module is either trivial or isomorphic to $V(\mathbf s)$ or $\Pi V(\mathbf s)$ for some typical
  regular $\mathbf s$.
  \end{corollary}
  \begin{proof} Recall the notations of Section 4. Consider a homomorphism $\mathbf C\to C_{\mathbf s}$ defined as the composition
 $$\mathbf C\hookrightarrow YQ(1) \buildrel\varphi_n\over\longrightarrow
 W^n\buildrel\theta_{\mathbf s}\over\longrightarrow C_{\mathbf s}.$$
 This homomorphism is surjective if $\mathbf s$ is typical regular, see Theorem \ref{regular}.
 For any central
    character $\chi$ there exists one up to isomorphism and parity change simple $\mathbf C_\chi$-module. By Proposition \ref{existence}
    it must be isomorphic to $V(\mathbf s)$.
    \end{proof}
    \begin{remark} If $\mathbf s=(s_1,\dots,s_n)$ and $\mathbf s'=(s_1,\dots,s_n,s,-s)$ then $V(\mathbf s)$ and $V(\mathbf s')$ admit the same central character.
      We can see it now from the formula
$$\frac{\sum_{i=0}^{\infty}\sigma_{2i+1}(\mathbf s')u^{-2i-1}} {1+\sum_{i=1}^{\infty}\sigma_{2i}(\mathbf s')u^{-2i}}=\frac{(1-s^2u^{-2})(\sum_{i=0}^{\infty}\sigma_{2i+1}(\mathbf s)u^{-2i-1})} {(1-s^2u^{-2})(1+\sum_{i=1}^{\infty}\sigma_{2i}(\mathbf s)u^{-2i})}.$$
      \end{remark}

    \begin{lemma}\label{commutators} We have the following expression
    \begin{align}\label{threeterms}
   &  [T_{1,1}^{(2k)},\eta_i] = [T_{1,1}^{(2k+2)},\eta_{i-2}] - [T_{1,1}^{(2k)},\eta_{i-1}] + 2T_{1,1}^{(2k)}\eta_{i-1}, \quad i\geq 2,\\
   & [T_{1,1}^{(2k)},\eta_0] = 2T_{1,-1}^{(2k)}, \quad
   [T_{1,1}^{(2k)},\eta_1] = -2T_{1,-1}^{(2k+1)} - [T_{1,1}^{(2k)},\eta_{0}] + 2T_{1,1}^{(2k)}\eta_{0}.
 \nonumber
\end{align}
 \end{lemma}

   \begin{proof} Note that according to (6.4) and (6.5) from
   \cite{PS2}
 $$[T_{1,1}^{(k)}, T_{1,-1}^{(1)}] = (1 + (-1)^k)T_{1,-1}^{(k)}.$$
 Hence $[T_{1,1}^{(2k)}, T_{1,-1}^{(1)}] = 2T_{1,-1}^{(2k)}$.
 Note also that
 \begin{equation}\label{ind1}
   [T_{1,1}^{(2)}, T_{1,-1}^{(2k+1)}] = 2T_{1,-1}^{(2k+2)},
   \end{equation}
\begin{equation}\label{ind2}
  [T_{1,1}^{(2)}, T_{1,-1}^{(2k)}] = 2T_{1,-1}^{(2k+1)} + 2T_{1,-1}^{(2k)} - 2T_{1,1}^{(2k)}T_{1,-1}^{(1)}.
  \end{equation}
 Using (6.9) from  \cite{PS2} we have that
 $[T_{1,1}^{(2)}, T_{1,1}^{(2k)}] = 0$. Hence
 $$[T_{1,1}^{(2k)},\eta_i] =
 (-\frac{1}{2})^i\operatorname{ad}^iT_{1,1}^{(2)}([T_{1,1}^{(2k)}, T_{1,-1}^{(1)}]) =
 \frac{(-1)^i}{2^{i-1}}\operatorname{ad}^iT_{1,1}^{(2)}(T_{1,-1}^{(2k)}).$$
 Next,
 $$\operatorname{ad}^iT_{1,1}^{(2)}(T_{1,-1}^{(2k)}) =
 \operatorname{ad}^{i-1}T_{1,1}^{(2)}([T_{1,1}^{(2)}, T_{1,-1}^{(2k)}]) =
 \operatorname{ad}^{i-1}T_{1,1}^{(2)}(2T_{1,-1}^{(2k+1)} + 2T_{1,-1}^{(2k)} - 2T_{1,1}^{(2k)}T_{1,-1}^{(1)}).$$
 Furthermore,
 $$2\operatorname{ad}^{i-1}T_{1,1}^{(2)}(T_{1,-1}^{(2k+1)}) =
 2\operatorname{ad}^{i-2}T_{1,1}^{(2)}([T_{1,1}^{(2)}, T_{1,-1}^{(2k+1)}]) =
 4\operatorname{ad}^{i-2}T_{1,1}^{(2)}(T_{1,-1}^{(2k+2)}) = (-1)^i2^{i-1}[T_{1,1}^{(2k+2)},\eta_{i-2}],$$
 $$2\operatorname{ad}^{i-1}T_{1,1}^{(2)}(T_{1,-1}^{(2k)}) = (-2)^{i-1}[T_{1,1}^{(2k)},\eta_{i-1}],$$
 $$-2\operatorname{ad}^{i-1}T_{1,1}^{(2)}T_{1,1}^{(2k)}T_{1,-1}^{(1)} =
 -2T_{1,1}^{(2k)}\operatorname{ad}^{i-1}T_{1,1}^{(2)}(T_{1,-1}^{(1)}) =
 -2T_{1,1}^{(2k)}((-2)^{i-1}\eta_{i-1}) = (-2)^iT_{1,1}^{(2k)}\eta_{i-1}.$$
 Thus
 $$[T_{1,1}^{(2k)},\eta_i] = \frac{1}{2^{i-1}}(2^{i-1}[T_{1,1}^{(2k+2)},\eta_{i-2}] -
 2^{i-1}[T_{1,1}^{(2k)},\eta_{i-1}] + 2^iT_{1,1}^{(2k)}\eta_{i-1}),$$ which gives (\ref{threeterms}).
\end{proof}
\begin{corollary}\label{fromcommutators} Let $\mathbf A$ be the commutative subalgebra in $YQ(1)$ generated by $T_{1,1}^{(2k)}$ for $k\geq 0$.
  Then $YQ(1)=\mathbf C\mathbf A=\mathbf A\mathbf C$.
\end{corollary}
\begin{proof} We will show that $\mathbf C\mathbf A\subset\mathbf A\mathbf C$. (The proof of the opposite inclusion is similar.)
  Let $D_i$ denote the span of $\eta_j$ for $j<i$. By Lemma \ref{commutators} for $i\geq 2$  we have that $\eta_iT^{(2k)}_{1,1}=T^{(2k)}_{1,1}\eta_i$ modulo
  $D_i\mathbf A+\mathbf AD_i$. Therefore, it suffices to show that $\eta_i\mathbf A\in \mathbf A\mathbf C$ for $i=0,1$. Furthermore,
  the relations in the second line of (\ref{threeterms}) imply that it suffices to show that $T_{1,-1}^{(m)}\in\mathbf C\mathbf A\cap \mathbf A\mathbf C$.
  This can be done by induction on $m$. The case $m=1$ is trivial as $\eta_0=T_{1,-1}^{(m)}$. For the step of induction if $m$ is even we employ (\ref{ind1})
  and if $m$ is odd (\ref{ind2}) and the relation
  $$[T_{1,1}^{(2)},\mathbf C]\subset\mathbf C.$$ Finally, since $\mathbf A$ and $\mathbf C$ generate $YQ(1)$ we get  $YQ(1)=\mathbf C\mathbf A=\mathbf A\mathbf C$.

  \end{proof}

Recall that  for any Hopf superalgebra $R$ the ideal $(R_1)$ generated by all odd elements is a Hopf ideal and the quotient $R/(R_1)$ is a Hopf algebra.
\begin{lemma}\label{quotient} The quotient $YQ(1)/(YQ(1)_1)$ is isomorphic to $\mathbf A\simeq\mathbb C[T_{1,1}^{(2k)}]_{k>0}$ with comultiplication
  $$\Delta T_{1,1}(u^{-2})=T_{1,1}(u^{-2})\otimes T_{1,1}(u^{-2}),$$
  where $T_{1,1}(u^{-2})=\sum T_{1,1}^{(2k)}u^{-2k}$.
\end{lemma}
\begin{proof} Since all $\eta_i$ generate $(YQ(1)_1)$,
  Lemma \ref{center} implies $YQ(1)=\mathbf A+(YQ(1)_1)$. Therefore there exists a surjective homomorphism $$\mathbf A\to YQ(1)/(YQ(1)_1).$$
  To prove that it is injective we need to show that $\mathbf A\cap (YQ(1)_1)=\{0\}$. It suffices to check that for any $y\in\mathbf A$ there exists a one-dimensional $YQ(1)$-module $\Gamma$ such that $y\Gamma\neq 0$. Let $y=P(T_{1,1}^{(2)}\dots T_{1,1}^{(2k)})$ for some polynomial  $P$ and consider the module
  $\Gamma=S(\mathbf t,0)$ as in Lemma \ref{one-dim}. Then $y$ acts on $\Gamma$ by $P(\sigma_1(\mathbf t),\dots,\sigma_k(\mathbf t))$. By a suitable choice of
  $\mathbf t$ we can get $P(\sigma_1(\mathbf t),\dots,\sigma_k(\mathbf t))\neq 0$.
  The comultiplication formula is straightforward as all $T^{(2k+1)}_{1,1}\in (YQ(1)_1)$.
\end{proof}
Let $f(u)=1+\sum_{k>0}f_{2k}u^{-2k}$. We denote by $\Gamma_f$ the
one-dimensional $\mathbf A$-module,  where
  the action of $T_{1,1}(u^{-2})$ is given by the generating function $f(u)$.
\begin{lemma}\label{one-dim-Y} The isomorphism classes of
one-dimensional $YQ(1)$-modules are in bijection with the set
$\{\Gamma_f\}$. Furthermore, we have the identity
  $\Gamma_f\otimes \Gamma_g\simeq \Gamma_{fg}$.
\end{lemma}
\begin{proof} Lemma \ref{quotient} reduces the statement to classification of one-dimensional $\mathbf A$-modules which is straightforward.
\end{proof}
\begin{theorem}\label{mainYangian} Any simple finite-dimensional $YQ(1)$-module is isomorphic to $V(\mathbf s)\otimes \Gamma_f$ or
  $\Pi V(\mathbf s)\otimes \Gamma_f$
  for some regular typical
  $\mathbf s$ and $f(u)=1+\sum_{k>0}f_{2k}u^{-2k}$. Furthermore, $V(\mathbf s)\otimes \Gamma_f$ and $V(\mathbf s')\otimes \Gamma_g$ are isomorphic up to
  change of parity if and only if $\mathbf s'$
  is obtained from $\mathbf s$ by permutation of coordinates and $f(u)=g(u)$.
  \end{theorem}
  \begin{proof} We start with regular typical $\mathbf s$ and identify $V(\mathbf s)$ with $V(\mathbf s)\otimes\Gamma_{\mathbf 1}$. Let $\chi$ be the central character of $V(\mathbf s)$ and consider only simple modules with  central character $\chi$. We denote by $YQ(1)^{\chi}$
    the quotient of $YQ(1)$ by the ideal generated by $\operatorname{Ker}\chi$. Note that $YQ(1)^{\chi}=\mathbf C_\chi\mathbf A$.

 Note that the central characters
    of $V(\mathbf s)$ and $V(\mathbf s)\otimes \Gamma_f$ are the same and they are isomorphic as $\mathbf C_\chi$-modules. For any finite-dimensional
    $YQ(1)$-module $M$ and $\theta(u)=1+\sum \theta_iu^{-2i}$
    set $M^{\theta}=\bigcap_{k>0,m>0} \operatorname{Ker}(T_{1,1}^{(2k)}-\theta_k)^m$.
    Clearly, we have an isomorphism of $\mathbf A$-modules
    $$M\simeq\bigoplus_{\theta\in P(M)} M^{\theta},$$
    for some finite set $P(M)$. Furthermore, we have the following obvious relations
    \begin{equation}\label{useful}
  P(M\otimes\Gamma_f)=P(M)f,\quad (M\otimes\Gamma_f)^{\theta f}=M^\theta \otimes\Gamma_f.
  \end{equation}
    This implies that $P(V(\mathbf s)\otimes \Gamma_f)=P(V(\mathbf s)\otimes \Gamma_g)$ if and only if $f=g$. Therefore we obtain the second assertion
    of the theorem.

    Consider the natural homomorphism
    $$F_\chi:YQ(1)^{\chi}\to\prod_f\End_{\mathbb C}(V(\mathbf s)\otimes \Gamma_f).$$

    \begin{lemma}\label{radical} Let $J_\chi=\ker F_\chi$. Then
      \begin{enumerate}
      \item $J_\chi=\mathbf{A}R_\chi=R_\chi \mathbf{A}$, where $R_{\chi}=\operatorname{Ann}_{\mathbf C_{\chi}}V(\mathbf s)$ is the Jacobson radical of $\mathbf{C}_\chi$;
        \item $J_\chi$ acts by zero on any simple finite-dimensional $YQ(1)^{\chi}$-module.
        \end{enumerate}
      \end{lemma}
      \begin{proof} Let us prove (1). Note that $T_{1,1}^{(2k)}$ acts on $V(\mathbf s)\otimes\Gamma_f$ as $\sum_{i=0}^k T_{1,1}^{2i}\otimes T_{1,1}^{2k-2i}$ and $\eta_0$ acts as $T_{1,-1}^{(1)}\otimes 1$. Therefore
by (\ref{etadef}) $\eta_i$ acts as $\eta_i\otimes 1$ for all $i\geq 0$ and hence
every $\zeta\in \mathbf C$ acts as $\zeta\otimes 1$. This implies $R_\chi\subset J_\chi$. Assume
        $$X=\sum_{i=0}^{k}\zeta_iT_{1,1}^{(2i)}\in J_\chi,\ \zeta_i\in\mathbf C_{\chi}.$$
          Set $f=1+u^{-2k}$. Then since $X$ annihilates both $V(\mathbf s)$ and $V(\mathbf s)\otimes\Gamma_f$ and $X$ acts on the latter module as $X\otimes 1+\zeta_{k}\otimes 1$ we obtain that $\zeta_k\in R_\chi$.
         Repeating this argument we obtain that all $\zeta_i\in R_\chi$.
       Thus, $J_\chi=R_{\chi}\mathbf{A}$. The equality  $\mathbf{A}R_\chi=R_\chi \mathbf{A}$ follows from $\mathbf{A}\mathbf{C}_\chi=\mathbf{C}_\chi \mathbf{A}$ by symmetry.

        To prove (2) note that $J_\chi=\mathbf{A}R_{\chi}$ annihilates the induced module $YQ(1)^{\chi}\otimes_{\mathbf C_{\chi}}V(\mathbf s)$ and hence any its quotient. On the other hand, up to switch of parity,
        any simple finite-dimensional
        $YQ(1)^{\chi}$-module is a quotient of this induced module. Hence the statement.
      \end{proof}

      \begin{corollary} Let $M$ be a finite-dimensional simple $YQ(1)^{\chi}$-module. Then $M$ is isomorphic to $V(\mathbf s)$ or $\Pi V(\mathbf s)$ for the regular typical $\mathbf s$ as a module over $\mathbf C_\chi$.
      \end{corollary}
      \begin{proof} The algebra $YQ(1)^{\chi}/J_\chi$ is a subalgebra in the product of matrix algebras of the size $\dim V(\mathbf s)$. Hence $\dim M\leq \dim V(\mathbf s)$.\footnote{ By the Amitsur-Levitzki identity and the Jacobson density theorem} Since $R_\chi$ annihilates $M$,
       the module $M$ is isomorphic to a direct sum of several copies of $V(\mathbf s)$ and $\Pi V(\mathbf s)$ as a module over $\mathbf C_\chi$. This implies the statement.
      \end{proof}
      \begin{remark} By Corollary \ref{bigclifford}, $\mathbf C_\chi/R_\chi\simeq C_{\mathbf s}$. Furthermore, $J_\chi\cap \mathbf C_\chi=R_\chi$.
        \end{remark}

    Denote by $\mathbf 1$ the function $\theta(u)=1$ and assume that $M$ is a simple finite-dimensional $YQ(1)^{\chi}$-module such that $M_0^{\mathbf 1}\neq 0$. Then $M$ is a quotient of the
    induced module $$I=(YQ(1)^\chi/J_\chi)\otimes_{\mathbf A}\Gamma_{\mathbf 1}.$$ Note that $$\dim I\leq\dim (\mathbf{C}_\chi/R_\chi)$$ but we will see later that the equality takes place.
    \begin{lemma}\label{induced} Let M be a simple $YQ(1)^\chi$-module such that  $M_0^{\mathbf 1}\neq 0$ and $M$ remains simple after restriction to
      $\mathbf C_\chi$. Then there exists a quotient $U$ of $I$ with all simple
      subquotients isomorphic to $M$ and length equal to  $\dim M_0^{\mathbf 1}$.
    \end{lemma}
    \begin{proof} Let $U=M\otimes (M_0^{\mathbf 1})^*$. It obviously has a filtration with all quotients isomorphic to $M$ and hence it satisfies the desired property.
      It remains to construct a surjective map $I\to U$. By Frobenius reciprocity we have a canonical isomorphism
      $$\Hom_{YQ(1)}(I,U)\simeq \Hom_{\mathbf A}(\Gamma_{\mathbf 1}, U)\simeq  \Hom_{\mathbf A}(\Gamma_{\mathbf 1},M^{\mathbf 1}\otimes (M_0^{\mathbf 1})^*).$$
      Consider the identity map in $\Hom_{\mathbf A}(\Gamma_{\mathbf 1},M^{\mathbf 1}\otimes (M_0^{\mathbf 1})^*)$ and denote by $\gamma$ the corresponding map
      in $\Hom_{YQ(1)}(I,U)$. Let us prove that $\gamma$ is surjective.
      First, observe that any $y\in \mathbf C$ acts on $M\otimes (M_0^{\mathbf 1})^*$ as $y\otimes 1$ by the same argument as
      in the proof of Lemma \ref{radical}.
     Choose a basis $\{v_1,\dots,v_r\}$ in $M_0^{\mathbf 1}$ and let
      $\{w_1,\dots,w_r\}$ be the corresponding dual basis in $(M_0^{\mathbf 1})^*$. By construction $\sum v_i\otimes w_i\in\im\gamma$. Since $M$ is a simple
      $\mathbf C_\chi$-module, by the Jacobson density theorem for every $i=1,\dots r$ there exists $y_i\in\mathbf C_\chi$ such that $y_iv_j=\delta_{i,j}v_1$.
      This implies $v_1\otimes w_i\in\im\gamma$ for all $i$ and hence $M\otimes w_i\in\im\gamma$ for all $i$. The surjectivity of $\gamma$ follows immediately.
    \end{proof}

    Now let us prove the first assertion of the theorem. Consider first the case $\mathbf s=(s_1,\dots,s_n)$ when $n$ is even. Then $\dim V(\mathbf s)=2^{n/2}$,
    $V(\mathbf s)$ is not isomorphic to  $\Pi V(\mathbf s)$ and $\dim (\mathbf C_\chi/R_\chi)=2^n$. By Lemma \ref{induced} and (\ref{useful}) for every $\theta\in P(V(\mathbf s))$ we have
    $$[I:V(\mathbf s)\otimes\Gamma_{\theta^{-1}}]\geq \dim V(\mathbf s)^\theta_0,\quad [I:\Pi V(\mathbf s)\otimes\Gamma_{\theta^{-1}}]\geq \dim V(\mathbf s)^\theta_1.$$
    On the other hand, $\dim I\leq\dim(\mathbf C_\chi/R_\chi)$. Hence any simple subquotient of $I$ is isomorphic to $V(\mathbf s)\otimes\Gamma_{\theta^{-1}}$ or
    $\Pi V(\mathbf s)\otimes\Gamma_{\theta^{-1}}$ and $\dim I=\dim(\mathbf C_\chi/R_\chi)$. Therefore every simple $YQ(1)^\chi$-module $M$ with $\mathbf 1\in P(M)$ is isomorphic to
    $V(\mathbf s)\otimes\Gamma_{\theta^{-1}}$ or
    $\Pi V(\mathbf s)\otimes\Gamma_{\theta^{-1}}$. If $f\in P(M)$ then $M$ is isomorphic to   $V(\mathbf s)\otimes\Gamma_{f\theta^{-1}}$ or
    $\Pi V(\mathbf s)\otimes\Gamma_{f\theta^{-1}}$. This implies the statement.

    Let us consider the case of odd $n$. Then $\dim V(\mathbf s)=2^{(n+1)/2}$,
    $V(\mathbf s)$ is isomorphic to  $\Pi V(\mathbf s)$ and $\dim (\mathbf C_\chi/R_\chi)=2^n$.  By Lemma \ref{induced}
    and (\ref{useful}) for every $\theta\in P(V(\mathbf s))$ we have
    $$[I:V(\mathbf s)\otimes\Gamma_{\theta^{-1}}]\geq \dim V(\mathbf s)^\theta_0=\dim  V(\mathbf s)^\theta_1.$$
    By counting dimensions we again obtain that every simple subquotient of $I$ is isomorphic to $V(\mathbf s)\otimes\Gamma_{\theta^{-1}}$. The end of the proof is
    the same as in the previous case.
    \end{proof}
    Let us conclude by stating the relation between $W$-modules and $YQ(1)$-modules.
    \begin{proposition} The simple $YQ(1)$-module $V(\mathbf s)\otimes \Gamma_f$ is lifted from some $W^{m+n}$-module if and only if $f\in\mathbb C[u^{-2}]$. Moreover, the smallest
      $m$ is equal to the degree of the polynomial $f$.
    \end{proposition}
    \begin{remark} Note that $m=2p$ is even. Then Theorem \ref{irreducible} and the diagram (\ref{diagram}) imply  $S(t_1,\dots,t_p,\lambda)\simeq V(\lambda)\otimes \Gamma_f$ where
      $$f=\prod_{i=1}^p(1+t_iu^{-2}).$$
      \end{remark}
    \begin{proof}Immediately follows from Theorem \ref{irreducible}.
      \end{proof}

  \section*{Acknowledgments}

\vskip 0.1in
\noindent
This work was supported by a grant from the Simons Foundation (\#354874, Elena Poletaeva) and the NSF grant (DMS-1701532, Vera Serganova).
We would like to thank V. G. Kac and M. L. Nazarov for useful comments.

\vskip 0.2in
\noindent
{\it Keywords:} Finite $W$-algebra, Queer Lie superalgebra, Yangian

\end{document}